\documentclass[11pt,a4paper]{amsart}
\usepackage{amssymb, amsthm, amsmath}
\usepackage{color}
\usepackage{xypic}
\usepackage{mathabx}
\usepackage{nicefrac}

\theoremstyle{theorem}
\newtheorem{theorem}{Theorem}[section]

\theoremstyle{corollary}
\newtheorem{corollary}[theorem]{Corollary}

\theoremstyle{lemma}
\newtheorem{lemma}[theorem]{Lemma}

\theoremstyle{example}
\newtheorem{example}[theorem]{Example}

\theoremstyle{definition}
\newtheorem{definition}[theorem]{Definition}

\theoremstyle{remark}
\newtheorem{remark}[theorem]{Remark}

\theoremstyle{proposition}
\newtheorem{proposition}[theorem]{Proposition}

\theoremstyle{claim}
\newtheorem{claim}[theorem]{Claim}

\newenvironment{Remark}{\begin{remark}\rm}{\end{remark}}

\def\id{\mbox{Id} }
\def\et{\quad\mbox{and}\quad}

\DeclareMathOperator{\GL}{GL}
\DeclareMathOperator{\SL}{SL}
\DeclareMathOperator{\Mat}{Mat}
\DeclareMathOperator{\pr}{pr}
\DeclareMathOperator{\im}{im}
\DeclareMathOperator{\rank}{rank}
\DeclareMathOperator{\Gr}{Gr}
\DeclareMathOperator{\F}{F}
\renewcommand{\id}{\textrm{id}}

\def\P{\mathbb{P}}
\def\PP{\mathbb{P}}

\def\C{\mathbb{C}}
\def\CC{\mathbb{C}}

\begin{document}

\title{Holomorphically equivalent algebraic embeddings}
\author{Peter Feller and Immanuel Stampfli}
\address{Boston College, Department of Mathematics, Carney Hall, Chestnut Hill, MA 02467, United States}
\email{peter.feller@alumni.unibe.ch}
\address{Jacobs University Bremen gGmbH, School of Engineering and Science, Department of Mathematics, Campus Ring 1, 28759 Bremen, Germany}
\email{immanuel.e.stampfli@gmail.com}

\subjclass[2010]{14R10, 32M17, 14J50}
\thanks{The authors gratefully acknowledge support by the Swiss National
Science Foundation Grants ``Positive Braids and Deformations" 155477
and ``Automorphisms of Affine Algebraic Varieties" 148627.}

\begin{abstract}
	We prove that two algebraic embeddings of a smooth variety $X$
	in $\CC^{m}$ are the same up to a holomorphic coordinate change,
	provided that $2 \dim X + 1$ is smaller than or equal to $m$.
	This improves an algebraic result
	of Nori and Srinivas. For the proof we extend a technique of
	Kaliman using generic linear projections of $\C^m$.
\end{abstract}
\maketitle
\section{Introduction}
In this paper we study algebraic embeddings of smooth algebraic varieties
(defined over $\C$) in $\C^m$. Given two embeddings $f$ and $g$ of an
algebraic variety $X$ in $\C^m$ it is natural to ask whether or not the two
embeddings are \emph{algebraically equivalent}, i.e.\ whether or not there exists
an algebraic automorphism $\phi$ of $\C^m$ such that $\phi\circ f=g$.
Only in special cases a full answer is known. Abhyankar, Moh and Suzuki
established that any two algebraic embeddings of $\C$ in $\C^2$ are
algebraically equivalent~\cite{AbMo_75}\cite{Su_74}. Nori
showed that any two algebraic embeddings of a smooth variety $X$ in $\C^m$
are algebraically equivalent as long as the following dimension condition holds:
$2\dim X+2\leq m$. This unpublished work of Nori was
generalized by Srinivas~\cite{Sri_91}.
In particular, for embeddings of $\C$ in $\C^m$,
the only unknown case is when $m$ equals $3$.
In fact, there are several examples of embeddings of $\C$ in $\C^3$  that
are conjectured to not be algebraically equivalent to the standard embedding;
for example, Shastri suggests examples coming from a knot-theoretic point of
view~\cite{Shastri_92}. For an overview of these embedding problems and how
they connect to the Cancallation Problem, the Linearization Problem,
and the Jacobian Conjecture in affine geometry, see Kraft and
van den Essen~\cite{Kra_96}\cite{Ess_04}.

Relaxing the condition of equivalence to the holomorphic setting,
Kaliman proved that any two algebraic embeddings $f$ and $g$ of
$\C$ in $\C^3$ are \emph{holomorphically equivalent}, i.e.\ there
exists a holomorphic automorphism (that is a biholomorphic self-map)
$\phi$ of $\C^3$ such that $\phi \circ f=g$~\cite{Ka_91}. In this article we
generalize this to the following result.

\begin{theorem}\label{thm:mainthm}
	Let $X$ be a smooth algebraic variety defined over $\C$ and
	$m \geq 2$ an integer.
	If $2\dim X+1\leq m$, then any two algebraic embeddings of $X$ in
	$\C^m$ are holomorphically equivalent.
\end{theorem}

Recently, it was pointed out to the authors that this Theorem is already proved
by Kaliman~\cite[Theorem~0.1]{Ka_13}.

\begin{remark}
It will be clear from our proof of Theorem~\ref{thm:mainthm} that the
auto\-morphism realizing the holomorphic equivalence can be chosen to
be a product of linear maps, and holomorphic shears and overshears.
\end{remark}

Theorem~\ref{thm:mainthm} can be seen as a one-dimensional
improvement of Nori's result under a relaxed, holomorphic equivalence condition.
As a special case we have that any embedding of $\C^n$ in $\C^{2n+1}$ is
holomorphically equivalent to the standard embedding
of $\C^n$ in $\C^{2n+1}$ given by
\[
    (x_1,\ldots,x_n)\mapsto
    (x_1,\ldots,x_n,0,\ldots,0) \, .
\]

As another corollary of Theorem~\ref{thm:mainthm},
any two algebraic embeddings of a smooth curve in $\C^3$ are
holomorphically equivalent. In contrast to this, almost nothing is known
about the algebraic equivalence of two embeddings of a smooth curve in $\C^3$.
Even if you are given two embeddings of a smooth curve in
$\C^2\times\{0\}\subset\C^3$,
which are seemingly simple embeddings from a three-dimensional point of view,
it is in general not known whether or not they are algebraically equivalent in $\C^3$;
compare Bhatwadekar and Roy~\cite{Bht_91},
where this question is asked in general and resolved for rational smooth curves.
Note that, while $\C$ has only one algebraic embedding in $\C^2$ up to 
algebraic equivalence, other smooth curves allow a lot 
of algebraically non-equivalent
embeddings into $\C^2$. For example, this is the case for algebraic
embeddings of $\CC^\ast$ in $\CC^2$; a partial
classification can be found in \cite{CKR_09}
and a full classification is announced in~\cite{KPR_14}.
In fact, these algebraic embeddings are often also holomorphically non-equivalent.
Indeed, the embeddings $f_{p, q} \colon \CC^\ast \hookrightarrow \CC^2$
given by $x\mapsto (x^p, x^{-q})$  for coprime positive integers $2\leq p<q$
are pairwise holomorphically non-equivalent.
The latter follows from the fact that these embeddings have different links at infinity,
see Neumann~\cite{Neumann_89}.

We note that the existence of holomorphically non-equivalent algebraic
embeddings of $\C^*$ in $\C^2$ shows that the dimension condition in
Theorem~\ref{thm:mainthm} cannot be improved if $\dim X=1$.
To our knowledge, for $\dim X >1$ there are no holomorphically
non-equivalent algebraic embeddings of a $n$-dimensional variety in
$\C^{2n}$ known, and so it remains an interesting question whether or not the
dimension condition in Theorem~\ref{thm:mainthm} can be improved.

We remark that from the holomorphic point of view it might seem natural to study
holomorphic (rather than algebraic) embeddings up to holomorphic equivalence.
However, in this setting the analogues of Theorem~\ref{thm:mainthm} and
Nori's result are known to be wrong. In fact, whenever $1\leq  n<m$ there
exist holomorphic embeddings of $\C^n$ in $\C^m$ that are not
holomorphically equivalent to the standard embedding
(see Forstneri{\v{c}}~\cite[Corollary 6.2]{For_99}). This makes the proof of
Theorem~\ref{thm:mainthm} subtle: we use a result that yields holomorphic not
algebraic equivalence, see Proposition~\ref{prop:kalitrick}; however, the projection
results we use only hold for algebraic (rather than holomorphic) embeddings,
compare Remark~\ref{remark:holomorphicProperness}.

\smallskip
We conclude the introduction with a short explanation on how the proof of
Theorem~\ref{thm:mainthm} works and how the paper is structured. We first
recall the following projection technique that Nori's result is based on.
Let $f$ and $g$ be
two algebraic embeddings of a smooth variety $X$ in $\CC^m$
and assume $2 \dim X +2 \leq m$. Srinivas (following Nori)
shows that there exists a
linear coordinate change of $\CC^m$ such that
any combination of $m-1$ coordinate functions coming from $f$ or $g$
yields an embedding from $X$  into $\CC^{m-1}$\cite[Theorem~1']{Sri_91}.
This can be used to construct $m+1$ algebraically equivalent embeddings of $X$
by successively replacing the coordinate functions of $f$ with coordinate functions
of $g$; thus, proving the algebraic equivalence of $f$ and $g$.
Roughly, if $m = 2 \dim X+1$, the idea of the proof
of Theorem~\ref{thm:mainthm} is to replace the condition that any combination
of $m-1$ coordinate functions of $f$ and $g$
is an embedding
by the weaker condition that this map is a proper immersion and its image
has only transversal double points; and then essentially to use an idea of
Kaliman, with which he proved that any two algebraic embeddings
of $\CC$ in $\CC^3$ are holomorphically equivalent. We outline the proof of
Theorem~\ref{thm:mainthm} in Section~\ref{sec:strategy}.
Section~\ref{sec:gpr} contains the main technical work of this paper.
There we provide results concerning generic projections, which deal with the
issues that arise from working in one dimension lower.
Finally, we combine these technical results to yield the proof of
Theorem~\ref{thm:mainthm} in Section~\ref{sec:proof}.
	
\section{Outline of the proof}\label{sec:strategy}
Let $f,g$ be two embeddings of a smooth $n$-dimensional variety
$X$ in $\C^{m}$ with $m\geq 2n+1$ and $n > 0$.
We want to show that $f$ and $g$ are
holomorphically equivalent.
We only consider the case $m=2n+1$ as otherwise the result
follows from Nori's work. Moreover, we can and will assume that every connected
component of $X$ has dimension $n$, see Lemma~\ref{lem:equidim}
in the Appendix.

We will inductively find a sequence $f_0=f$, $f_1$, $\ldots$ , $f_{2n}$,
$f_{2n+1}=g$ of $m$ algebraic
embeddings such that two consecutive
embeddings are holomorphically equivalent.
Concretely, for some
$0 \leq l < 2n+1$, assume that $f_l$ and a linear projection
$r_l \colon \CC^{2n+1} \to \CC^l$ have been constructed
such that the last $l$ coordinate functions of $f_l$ equal $r_l \circ g$ . The crucial
step is to construct an algebraic embedding
$f_{l+1} \colon X \to \C^{2n+1}$ that is holomorphically equivalent to $f_l$ and a
linear projection $ r_{l+1} \colon \CC^{2n+1} \to \CC^{l+1}$
such that the last $l+1$ coordinate functions of $f_{l+1}$ equal $r_{l+1} \circ g$.
From this, one gets the desired sequence of embeddings.

In the sequel we describe this crucial step in more detail. For this purpose we
introduce some terminology.
Let $h \colon X \to \CC^k$ be a holomorphic map. 	
A pair $(x, y)$ in $X \times X$
with $x \neq y$ is called a \emph{double point of $h$}, if
$h(x) = h(y)$. The double point is called \emph{transversal}, if
the images of the differentials
$D_x h$ and $D_y h$ span the whole tangent space
of $\CC^k$ in $h(x) = h(y)$.
If all double points of $h$ are transversal,
we say that $h$
is \emph{transversal}. If in addition every fiber of $h$ has at most
$j$ elements, then we say it is \emph{$j$-transversal}.

The following proposition is the only holomorphic ingredient in our proof.
	
\begin{proposition}\label{prop:kalitrick}
	Let $f \colon X \to \CC^{2n+1}$ be an algebraic
	embedding of a smooth
	$n$-dimensional variety and let $p \colon \CC^{2n+1} \to \CC^{2n}$
	be a linear projection
	such that $p \circ f$ is proper, 2-transversal and immersive.
	If $d \colon X \to \CC$ is an
	algebraic function such that $d(s_1) \neq d(s_2)$ for all
	double points $(s_1, s_2)$ of $p \circ f$, then
	\[
		X \to \CC^{2n+1} \, , \
		x \mapsto (p \circ f(x), d(x))
	\]
	is an algebraic embedding that is holomorphically equivalent to $f$.
\end{proposition}

This proposition generalizes an idea of Kaliman,
see \cite[Theorem 6]{Ka_91} and its proof.

\begin{Remark}
\label{Rem:kalitrick}
It is worth mentioning that the proof
of Proposition~\ref{prop:kalitrick} carries
over to a holomorphic setting, i.e.~when $f$ and $d$ are holomorphic rather
than algebraic maps. To see this, in the proof below one only needs to replace
all algebraic maps by holomorphic ones and to notice that the image of
the double points of $p\circ f$ is a discrete closed subset of $\C^{2n}$.
\end{Remark}

\begin{proof}[Proof of Proposition~\ref{prop:kalitrick}]
	Let $h \colon \CC^{2n+1} \to \CC$ be a linear function such that
	$f = (p \circ f, h)$.
	Let $S \subseteq \CC^{2n}$ be the finite set of
	points $s \in \CC^{2n}$
	such that $(p \circ f)^{-1}(s) = \{ s_1, s_2 \}$
	with $s_1 \neq s_2$. For every $s$ in $S$,
	it holds that $h(s_1) \neq h(s_2)$ since $f$ is an embedding.
	Thus, we can choose a number $a_s \in \CC$ such that
	\[
		d(s_1) - e^{a_s} h(s_1) = d(s_2) - e^{a_s} h(s_2) \, .
	\]
	Let $\psi_1 \colon \CC^{2n} \to \CC$ be an algebraic
	function such
	that for all $s \in S$ we have $\psi_1(s) = a_s$ and let
	$\alpha_1$ be the holomorphic automorphism of
	$\CC^{2n} \times \CC$
	defined by $\alpha_1(w, \lambda) = (w, e^{\psi_1(w)}\lambda)$.
	By composing $f$ with $\alpha_1$, we can assume that
	$d(s_1) - h(s_1) = d(s_2) - h(s_2)$ for all $s \in S$.
	Note that in general
	$h$ is no longer algebraic. However, it is holomorphic.
	Let $\Gamma \subset \CC^{2n}$
	be the image of $p\circ f \colon X \to \CC^{2n}$. By
	Remmert's proper mapping theorem
	\cite[Satz~23]{Re_57},
	$\Gamma$ is a closed (analytic) subvariety of $\CC^{2n}$.
	Now, using that $p \circ f$ is immersive and
	2-transversal, we get a holomorphic factorization
	\[	
		\xymatrix{
			X \ar[rd]_-{d-h} \ar[r]^-{p \circ f} &
			\Gamma \ar@{.>}[d]^-e \\
			& \CC \ar@{}[r]|-{ \textrm{\normalsize.}} &
		}
	\]
	Using Cartan's Theorem~B
	\cite[Th\'eor\`eme~B]{Car_53}, we can
	extend $e$ to a holomorphic map $\psi_2 \colon \CC^{2n} \to \CC$.
	Let $\alpha_2$ be the holomorphic automorphism
	of $\CC^{2n} \times \CC$ defined by
	$\alpha_2(w, \lambda) = (w, \lambda + \psi_2(w))$.
	Thus, $\alpha_2 \circ f(x) = (p \circ f(x), d(x))$ and this map is algebraic.
\end{proof}
Proposition~\ref{prop:kalitrick} is used in the crucial step as follows.
We show the existence of a linear projection $p \colon \CC^{2n+1} \to \CC^{2n}$
such that $p \circ f_l$
is proper, 2-transversal, and immersive, and
such that the last $l$ coordinate functions of $p \circ f_l$ and $f_l$ are the same;
this constitutes the main technical work of our proof.
In fact, to find $p$
we have to compose $f_l$ with an algebraic auto\-morphism
of $\CC^{2n+1}$ that fixes the last $l$ coordinates.
Then we construct a surjective linear function
$e \colon \CC^{2n+1} \to \CC$ such that $e \circ g$ distinguishes the
double points of $p \circ f_l$. Now, Proposition~\ref{prop:kalitrick} implies
that
\[
	f_{l+1} = (p \circ f_l, e \circ g)  \colon X \to \CC^{2n} \times \CC
\]
is holomorphically equivalent to $f_l$. If we define
$r_{l+1} \colon \CC^{2n+1} \to \CC^{l} \times \CC$ by $r_{l+1}(z) = (r_l(z), e(z))$,
then the last $l+1$ coordinate functions of $f_{l+1}$ equal $r_{l+1}\circ g$.

\smallskip
Of course, the proof of our main theorem cannot
carry over to a holomorphic setting by the existence of
non-equivalent holomorphic
embeddings of $\CC^n$ in $\CC^{2n+1}$. While
Proposition~\ref{prop:kalitrick} remains true in the holomorphic setting
(see Remark~\ref{Rem:kalitrick}),
a serious issue is that for
generic linear projections $p \colon \CC^{2n+1} \to \CC^{2n}$
the map $p \circ f \colon X \to \CC^{2n}$ need not be proper if
$f \colon X \to \CC^{2n+1}$ is a holomorphic embedding
(see Remark~\ref{remark:holomorphicProperness}).

In Section~\ref{sec:proof} we provide the details of the proof,
based on several projection results.

\section{Generic projection results}\label{sec:gpr}
In this section we present projection results that are needed for our proof
of Theorem~\ref{thm:mainthm}. More precisely, they are needed to
guarantee that certain projections behave well in the following sense.

\begin{definition}
	Let $f \colon X \to \CC^{2n+1}$ be an embedding.
	A surjective linear map $p \colon \CC^{2n+1} \to \CC^{2n}$ is called a
	\emph{good projection} for $f$,
	if the composition $p \circ f$ is proper, immersive and $2$-transversal.
\end{definition}		

The moral of this section is, that generic projections are good projections and
that this holds in various
relative settings.
Before we start with these results, let us set up some terminology
and notation.

\smallskip

{\bf Terminology and notation.}
Let $Z$ be a variety.
We say that a statement is true for \emph{generic} $z \in Z$
if there exists a dense Zariski open subset $U \subseteq Z$ such
that the statement is true for all $z \in U$.
Maps between algebraic varieties are understood to be
algebric morphisms if not stated otherwise.

Let $0 \leq l \le m$ be integers. We denote the $m$-dimensional projective space
by $\PP^m$
and the Grassmannian of $l$-dimensional linear subspaces of $\C^m$ by
$\Gr^m_l$. The map
$p_l\colon \C^{m}\to \C^l$ denotes the standard projection to the last
$l$ coordinates and $\P^{m-l} \subseteq \P^{m}$
denotes the linear projective subspace
given by setting the last $l$
projective coordinates of $\P^m$ to zero.
If $V \subseteq \CC^{m}$ is a linear subspace, then
$p_V \colon \CC^m \to \CC^m / V$ denotes the canonical linear projection.
For $v \in \PP^{m-1}$, note that $p_l \colon \CC^m \to \CC^l$ factors through
 $p_v \colon \CC^m \to \CC^{m} /v$ if and only if $v \in \PP^{m-1-l}$.

\smallskip

In a first subsection, we gather results concerning transversal maps.
Afterwards, we divide the projection results into three subsections corresponding
to the three properties of a good projection.

Let us give a flavour of the
type of results we are looking for. Consider an algebraic embedding
$f \colon X \to \CC^{2n+1}$
of a smooth $n$-dimensional variety
$X$. We prove that for generic $v \in \PP^{2n}$, the map
$p_v \colon \CC^{2n+1} \to \CC^{2n+1} / v {=} \CC^{2n}$
is a good projection for $f$.
This is done in Corollary~\ref{corollary:baselproper},
Remark~\ref{rem:immersion} and Lemma~\ref{lemma:doublepointsllessn}.
Now, as mentioned in the outline,
for the inductive construction of the embedding $f_{l+1}$ from $f_l$
we need good projections $p_v$ for $f_l$ that ``remember"
the last $l$ coordinates, i.e.\ the projection
$p_l \colon \CC^{2n+1} \to \CC^l$ factors through $p_v$.
In other words, we would like
for generic $v \in \PP^{2n-l}$
that $p_v$ is a good projection for $f_l$ (where we view $\PP^{2n-l}$
as a linear projective subspace
of $\PP^{2n}$). Of course, this is only possible
under certain restrictions on the projection to the last
$l$ coordinate
functions $p_l \circ f \colon X \to \CC^l$.
These results are the content
of Corollary~\ref{corollary:baselproper},
Lemma~\ref{lemma:goodisgenericimmersion},
Lemma~\ref{lemma:doublepointsllessn} and
Lemma~\ref{lemma:doublepointslbiggern}.

\subsection{Transversal maps}

Throughout this subsection we fix a closed smooth
$n$-dimensional subvariety $X \subset \CC^{2n+1}$.
For any map $h \colon X \to \CC^{k}$, we consider
the map
\[
	\psi_h \colon X^2 \to \CC^{k} \, , \
		(x,y) \mapsto h(y)-h(x)
\]
and its composition with the canonical projection
$\pi \colon \CC^{k} \setminus \{ 0 \} \to \PP^{k-1}$
\[
	\overline{\psi}_h =
	\pi \circ \psi_h \colon X^2 \setminus \psi_h^{-1}(0) \to \PP^{k-1} \, .
\]
When $h$ is the fixed embedding $X \subset \C^{2n+1}$,
we denote $\psi_h$ by $\psi$ and $\overline{\psi}_h$ by $\overline{\psi}$.

\begin{lemma}
	\label{transkey.lem}
	If $h \colon X \to \CC^{k}$ is transversal,
	then $p_v \circ h$ is transversal for generic $v \in \PP^{k-1}$.
	In fact, $p_v \circ h$ is transversal if and only if
	$v$ is a regular value of $\overline{\psi}_h$.
\end{lemma}

 \begin{proof}[Proof of Lemma~\ref{transkey.lem}]
 Let $x \neq y$ in $X$ such that $h(x) = h(y)$. By assumption, $(x, y)$
 is then a transversal double point of $h$ and thus also of
 $p_v \circ f$, for all $v \in \PP^{k-1}$.

Now, fix $v$ in $\P^{k-1}$ and let $p = p_v \colon \CC^{k} \to \CC^{k} / v$.
One sees immediately that for $x, y \in X$ with $h(x) \neq h(y)$,
having the same image under $p \circ h$ is equivalent to $\overline{\psi}_h(x,y)=v$.
As for a given $w \in \C^{k} \setminus \{ 0 \}$ with $w \in v$,
we have $\ker(D_w \pi) = v$,
the differentials $D_w\pi$ and $D_w p = p$ are the same
up to a linear isomorphism between the targets.
Thus, for $x,y \in X$ with $0 \neq h(y)-h(x)\in v$ the differentials at $(x,y)$ of
\[
	\overline{\psi}_h \colon X^2\backslash \psi_h^{-1}(0)\to \P^{k-1}
	\et
	p \circ \psi_h
	\colon X^2\backslash\psi_h^{-1}(0) \to \C^{k}
\]
have the same rank.
Therefore, the statement follows from the fact that the image of
$D_{(x,y)} (p \circ \psi_h)$ is $\im D_{x}(p\circ h)+\im D_{y}(p\circ h)$.
\end{proof}

\begin{remark}\label{rem:psiisdominant}
After applying an algebraic automorphism of $\CC^{2n+1}$ to $X$,
one can assume that $\overline{\psi}$ is dominant.
This is seen as follows. We fix some $x \neq y$ in $X$ and study the linear
span $V_X$ of $x-y$, $T_x X$ and $T_y X$ inside $\C^{2n+1}$.
Note that $V_X=\C^{2n+1}$ if and only if $(x,y)$ is a regular point of
$\overline{\psi}$, by the proof of Lemma~\ref{transkey.lem}.
Therefore, it suffices to find an algebraic automorphism $\varphi$ of
$\C^{2n+1}$ such that $\varphi(x) = x$, $\varphi(y) = y$ and
$V_{\varphi(X)}=\C^{2n+1}$. However, such a $\varphi$ exists since
one can prescribe the derivative of an algebraic automorphism of
$\CC^{2n+1}$ in finitely many fixed points,
see Forstneri\v{c}~\cite[Corollary~2.3]{For_99} (citing Buzzard~\cite{Buz_98}).
In fact, in the case we need,  $\varphi$ can be chosen to be a composition
of algebraic shears and linear maps. For completeness, we provide
a suitable version of this result in the Appendix,
see Lemma~\ref{lem:prescribedjets}.
\end{remark}

\begin{lemma}\label{lemma:transversalmultipointsdescent}
Let $0 \leq l \leq 2n$ be an integer. If  $p_l |_X$
is transversal, then $p_v |_X$ is transversal for generic
$v \in \PP^{2n-l}\subseteq \PP^{2n}$.
\end{lemma}

 \begin{proof}
 Denote by $\Delta$ the diagonal in $X^2$.
 Let $C_{2n-l} \subseteq X^2 \setminus \Delta$ be the preimage of
 $\PP^{2n-l}$ under $\overline{\psi}$ and denote the restriction of $\overline{\psi}$
 to $C_{2n-l}$ by
 \[
 	\overline{\psi}' = \overline{\psi} |_{C_{2n-l}} \colon C_{2n-l}       \to \PP^{2n-l} \, .
 \]
 For $(x, y) \in C_{2n-l}$, consider the following linear map
  \[
  	\xymatrix@C=20pt{
 		\nu_{(x, y)} \colon T_{(x, y)} X^2 \setminus \Delta
		\ar[rr]^-{D_{(x, y)} \overline{\psi}} &&
		T_{\overline{\psi}(x, y)} \PP^{2n} \ar[r] &
		T_{\overline{\psi}(x, y)} \PP^{2n} /
		T_{\overline{\psi}(x, y)} \PP^{2n-l}
		{=} \CC^l \, .
	}
 \]
 For generic $v \in \PP^{2n-l}$, the differential
 $D_{(x, y)} \overline{\psi}'$ has full rank for all
 $(x, y) \in (\overline{\psi}')^{-1}(v)$ by an algebraic geometry version of Sards
 theorem\footnote{For lack of reference, we provide the needed version of
 Sard's theorem in the Appendix, see Lemma~\ref{lem:GenericSurj}.}.
 Thus, by the second part of Lemma~\ref{transkey.lem},
 it is enough to show that $\nu_{(x, y)}$
 is surjective for all $(x, y) \in C_{2n-l}$.
 Let $V_i \subset \PP^{2n-l}$ be the open affine subset of points such that the
 $i$-th coordinate does not vanish and let us denote by $h_i \colon X \to \CC$
 the $i$-th coordinate function of the embedding $X \subset \CC^{2n+1}$.
A calculation shows that for
 $(x, y) \in (\overline{\psi})^{-1}(V_i)$ the linear map $\nu_{(x, y)}$ is given by
 \[
 	T_{(x, y)} X^2 \setminus \Delta \to \CC^l \, , \
 	(w_1, w_2) \mapsto \frac{1}{h_i(y) - h_i(x)}
 	\big( (D_y p_l |_X) w_2 - (D_x p_l |_X ) w_1 \big) \, .
 \]
 As by assumption $p_l |_X$ is transversal, it follows for all
 $(x, y) \in X^2 \setminus \Delta$ that
 \[
 	T_{(x, y)} X^2 \setminus \Delta \to \CC^l \, , \
 	(w_1, w_2) \mapsto (D_y p_l |_X) w_2 - (D_x p_l |_X) w_1
 \]
 is surjective. This implies that
 $\nu_{(x, y)}$ is surjective for all
 $(x, y) \in (\overline{\psi})^{-1}(V_i)$.
 \end{proof}

The next lemma generalizes Lemma~\ref{transkey.lem}. It will be formulated in terms of flags. Denote by $\F_{2n+1}$ the flag variety of
$\C^{2n+1}$, i.e.\ the closed subvariety of the following product of Grassmannians
\[
	\Gr^{2n+1}_0 \times
	\Gr^{2n+1}_1 \times \cdots \times \Gr^{2n+1}_{2n+1}
\]
given by the condition that $(V_0,V_1,\ldots, V_{2n+1})$ satisfies
$V_i\subseteq V_{i+1}$.
Let us view $\GL_{2n+1}$ as an open subset of $(\C^{2n+1})^{2n+1}$ and let
$\eta \colon \GL_{2n+1} \to \F_{2n+1}$ be given by
\[
	\eta(v_1,v_2,\ldots,v_{2n+1})
	= \left(0, [v_1] , [v_1,v_2],\ldots,[v_1,v_2,\ldots,v_{2n+1}] \right)
\]
where $[v_1, \ldots, v_i]$ denotes the linear
span of $v_1, \ldots, v_i$ in $\CC^{2n+1}$.
Consider the $(2n+1)$-fold product of the map $\psi \colon X^2 \to \CC^{2n+1}$
\[
	\phi = \psi \times \cdots \times \psi
	\colon (X^2)^{2n+1} \to (\C^{2n+1})^{2n+1}
\]
and let $S = \phi^{-1}(\GL_{2n+1}) \subseteq (X^2)^{2n+1}$.
We define $\overline{\phi} \colon (X^2)^{2n+1} \dashrightarrow \F_{2n+1}$ as the
rational map $\eta \circ \phi \colon (X^2)^{2n+1} \dashrightarrow \F_{2n+1}$.
Note that $S$ is the set where $\overline{\phi}$ is defined.

\begin{lemma}
\label{lemma:flag}
Let $s=((x_1,y_1),\ldots,(x_{2n+1},y_{2n+1})) \in S$
and let the flag $(0 \subset W_1 \subset \cdots \subset \CC^{2n+1})$
be the image of $s$ under $\overline{\phi}$. Then, $s$
is a regular point of $\overline{\phi}$ if and only if $(x_i,y_i)$
is a transversal double point of $p_{W_i} |_X$ for
all $1 \leq i \leq {2n+1}$.
\end{lemma}

 \begin{proof}

 Let $a = \phi(s) \in \GL_{2n+1}$. We claim that the differential
 \[
 	D_a \eta  \colon T_a \GL_{2n+1} \to T_{\eta(a)} \F_{2n+1}
 \]
 is the same as the projection
 \[
 	p_{W_1} \times \cdots \times p_{W_{2n+1}}
	\colon  (\CC^{2n+1})^{2n+1}\to
	\C^{{2n+1}}/W_1 \times\cdots
	\times\C^{2n+1}/W_{{2n+1}}
 \]
 (up to a linear isomorphism of the targets).
 Hence, $s$ is a regular point of $\overline{\phi}$ if and only if $(x_i,y_i)$
 is a regular point of the map $p_{W_i} \circ \psi \colon X^2 \to \C^{m}/W_i$
 for all $1\leq i\leq {2n+1}$.
 This finishes the proof since $(x,y)$ is a regular point of
 $p_{W_i} \circ \psi$ if and only if $(x,y)$
 is a transversal double point of $p_{W_i} |_X$.

 It remains to prove that $D_a \eta$ and
 $p_{W_1} \times \cdots \times p_{W_{2n+1}}$ are the same as claimed above.
 Noting that $ T_{\eta(a)} \F_{2n+1}$ and
 $\C^{2n+1}/W_1\times\C^{2n+1}/W_2\times\cdots\times\C^{2n+1}/W_{2n+1}$ have
 the same dimension, we have to show that
 $D_a \eta$ and  $p_{W_1} \times \cdots \times p_{W_{2n+1}}$ have the same
 kernel, i.e.\
 \[
 	\ker D_a \eta = W_1 \times W_2 \times \cdots \times
	W_{2n+1} \subset (\CC^{2n+1})^{2n+1}.
 \]
 Let $\varphi_a \colon \GL_{2n+1} \to \GL_{2n+1}$ be given by
 $\varphi_a(h) = ah$.
 Then, the differential $D_e \varphi_a$ identifies with
 $\Mat_{2n+1} \to \Mat_{2n+1}$, $M \mapsto aM$. Let
 $B \subseteq \GL_{2n+1}$ be the subgroup of upper triangular matrices.
 Viewing the flag variety $\F_{2n+1}$ as the quotient of
 $\GL_{2n+1}$ by right multiplication with $B$,
 the map $\eta$ identifies with the quotient map.
 Clearly, $\ker D_e \eta \subset\Mat_{2n+1}$
 are the upper triangular matrices. Thus,
 $\ker D_a \eta = D_e \varphi_a (\ker D_e \eta) \subset \Mat_{2n+1}$ consists
 of those matrices $M$ such that $a^{-1} M$ is upper triangular, i.e.~by
 letting $m = 2n+1$ and $a = (a_1 \ \cdots \ a_m)$
 	\begin{eqnarray*}
 		 M &=& \begin{pmatrix}
 		 a_1 & \cdots & a_{m}
 		      \end{pmatrix}
 		      \begin{pmatrix}
 			 b_{11} & \cdots & b_{1m} \\
 			 \vdots & \ddots & \vdots \\
 			 0 & \cdots  & b_{mm} \\
 		      \end{pmatrix} \\
 		      &= &
 		      \begin{pmatrix}
 		      	   b_{11} a_1 & b_{12} a_1 + b_{22} a_2 & \cdots &
 		      	   b_{1m} a_1 + \ldots + b_{mm} a_{m}
 		      \end{pmatrix} \\
 		      & \in & W_1 \times W_2 \times \cdots \times W_{m} \, .
 	\end{eqnarray*}
This proves the lemma.
 \end{proof}

\begin{Remark}
\label{rem:dominancPhi}
After applying an algebraic automorphism of $\CC^{2n+1}$ to $X$,
one can assume that
$\overline{\phi} \colon (X^2)^{2n+1} \dashrightarrow \F_{2n+1}$
is dominant. This is seen as follows.
By Remark~\ref{rem:psiisdominant},
we can assume that $\overline{\psi}$ is dominant.
This allows us to pick a basis $(v_1,\ldots,v_{2n+1})$ of
$\C^{2n+1}$ such that
for all $i$ we have that $[v_i]\in\P^{2n}$
is a regular value of $\overline{\psi}$ and it is the image of some
$(x_i,y_i) \in X^2$ under $\overline{\psi}$.
Then $((x_1,y_1),\ldots,(x_{2n+1},y_{2n+1}))$
is a regular point of $\overline{\phi}$, and so the statement follows.
\end{Remark}
As a consequence of Lemma~\ref{lemma:flag}
and Remark~\ref{rem:dominancPhi} we get the following.

\begin{corollary}
	\label{corollary:genericflags}
	After applying an algebraic automorphism of $\CC^{2n+1}$ to $X$,
	the maps $p_{W_1} |_X$, \ldots, $p_{W_{2n}} |_X$ are transversal
	for generic flags
	$(0 \subset W_1 \subset \cdots \subset W_{2n} \subset \CC^{2n+1})$
	in $\F_{2n+1}$.
\end{corollary}

\subsection{Proper maps}
Let $1 \le k < m$ be integers and let $X \subseteq \CC^{m}$ be
a closed subvariety.
We consider the following subset of the Grassmannian $\Gr_k^m$
\[
	\omega_k(X) = \{ V \in \Gr^m_k | \, \exists \, \{ x^j \} \subseteq X \, , \,
					\lim_{j \to \infty} \pi(x^j) \subseteq V  \, ,
					\, | x^j | \to \infty \} \, .
\]
where $\pi \colon \CC^{m} \setminus \{ 0 \} \to \PP^{m-1}$ is the canonical
projeciton. In the case $k=1$, $\omega_1(X)$ is a closed analytic subset of
$\PP^{m-1} = \Gr_1^m$ by \cite[Corollary~5.1]{Chi_89}
and thus it is Zariski closed by Chow's Theorem \cite[Theorem~V]{Cho_49}).
Moreover, $\dim \omega_1(X) < \dim X$ and
$p_v |_X \colon X \to \CC^{m} / v$ is proper for all
$v \in \PP^{m-1} \setminus \omega_1(X)$ by \cite[Lemma~4.11.2]{Fo_11}.
This yields the following.

\begin{corollary}
	\label{corollary:baselproper}
	Let $m = 2n+1$ and let $X \subset \CC^{2n+1}$
	be $n$-dimensional. If $0 \leq l \leq n$, then $p_v |_X$ is proper for
	generic $v \in \PP^{2n-l}\subseteq \PP^{2n}$.
\end{corollary}

For arbitrary $k$ this generalizes to the following.

\begin{lemma}\label{lemma:baselproper}
	If $1 \leq k < m$, then
	$\dim \omega_k(X) < \dim \Gr^m_k - (m - k) + \dim X$
	and $p_V |_X \colon X \to \CC^m / V$ is proper for all
	$V \in \Gr^m_k \setminus \omega_k(X)$.
\end{lemma}

 \begin{proof}
 	Let $E = \{ \, (W, V) \in \omega_1(X) \times \Gr^m_k \ | \ W \subseteq V \, \}$.
 	Note that $\omega_k(X)$ is the image of the projection
 	$E \to \Gr^m_k$ to the second factor.
 	Projection to the first factor $E \to \omega_1(X)$ turns $E$ into
 	a fiber bundle with fiber-dimension $(m-k)(k-1) = \dim \Gr^m_k -(m-k)$.
 	This implies the dimension inequality.
 	Let $V \in \Gr^m_k$.
 	If $p_V |_X$ is not proper, then there exists
 	a sequence $ \{ x^j \} \subseteq X$, such that $| x^j | \to \infty$
 	and $p_V(x^j)$ is bounded. Without loss of generality, we can assume
 	that $\lim \pi(x^j) \in \PP^{m-1}$ exists. However, this implies
 	$\lim \pi(x^j) \subseteq V$ and thus $V \in \omega_k(X)$.
 \end{proof}

\begin{remark}
	\label{remark:holomorphicProperness}
	In a holomorphic setting Lemma~\ref{lemma:baselproper} and
	Corollary~\ref{corollary:baselproper} do not hold. Indeed,
	if $X$ is an analytic subvariety
	(rather than an algebraic subvariety),
	then the dimension of $\omega_1(X)$
	is no longer strictly smaller than the dimension of $X$
	(compare the comments before~\cite[Lemma~4.11.2]{Fo_11}).
	For example, for
	\[
		X=\{\,(t,e^t) \ |\ t\in\C\, \}\subset\C^2,
	\]
	the subset $\omega_1(X)\subseteq\P^1$ is equal to $\P^1$.
	More generally,
	there are holomorphic embeddings $f \colon \CC \to \CC^m$
	for any $m > 1$ such that $\omega_1(f(X)) = \PP^{m-1}$
	\cite[Proposition~2]{FR_96}.
\end{remark}

\subsection{Immersive maps}
Fix a closed smooth $n$-dimensional subvariety $X \subset \CC^{2n+1}$.
For any map $h \colon X \to \CC^k$ and any
integer $0 \leq i \leq k$, we denote
\[
	X_i = X_i(h) = \{ \, x \in X \ | \ \rank D_x h  = i \, \} \, .
\]
Hence, $\bigcup_{i=0}^k X_i$ is a partition of $X$
where every $X_i$ is locally closed in $X$.

\begin{lemma}\label{lemma:baselimmersion}
	Let $1 \leq k \leq 2n$ and let
	$0 \leq i < n$. Then for generic  $V \in \Gr^{2n+1}_k$ we have
	$\dim X_i(p_V |_X) \leq k + i-n-1$.
\end{lemma}

 \begin{proof}[Proof of Lemma~\ref{lemma:baselimmersion}]
 	We can assume that $n-i \leq k \leq 2n+1-i$, as otherwise
 	$X_i(p_V |_X)$ would be empty.
	We denote by $\Gr_j(TX)$ the bundle over $X$
	of \mbox{$j$-dimensional} linear subspaces of the tangent spaces of $X$.
 	Let
	\[
		E = \{ \, (W, V) \in \Gr_{n-i}(TX) \times \Gr^{2n+1}_k \ | \
		W \subseteq V \, \}
	\]
 	and let $E_V$ be the fiber over $V \in \Gr^{2n+1}_k$ under
 	the projection onto the second factor $E \to \Gr^{2n+1}_k$.
 	The projection $E \to \Gr_{n-i}(TX) \to X$ yields a surjective map
 	\[
 		E_V \to \{ \, x \in X \ | \ \dim T_x X \cap V \geq n-i \, \} \, .
 	\]
 	This implies $\dim X_i(p_V |_X) \leq \dim E_V$.
 	Projection to the
 	first factor turns $E$ into a
 	fiber bundle over $\Gr_{n-i}(TX)$ with fiber-dimension
 	$(2n+1-k)(k-n+i)$. Thus, we have
 	\[
 		\dim E = (2n+1-k)k + (n-i)(k+i-2n-1) + n \, .
 	\]
 	Hence, either $E_V$ is empty for generic $V \in \Gr_k^{2n+1}$
	or we have
	\[
		\dim E_V = (n-i)(k+i-2n-1) + n \geq 0 \
		\textrm{for generic $V \in \Gr_k^{2n+1}$} \, .
	\]
 	Since $k \leq 2n+1-i$ and $i < n$,
	we get $(n-i)(k+i-2n-1) + n \leq k+i-n-1$. This implies the result.
 \end{proof}

\begin{lemma}\label{lemma:goodisgenericimmersion}
Let $0 \leq l \leq 2n$. If $\dim X_i(p_l |_X) \leq n+i-l$
for all $0 \leq i < n$, then $p_v |_X$ is immersive for generic
$v \in \PP^{2n-l}\subseteq \PP^{2n}$.
\end{lemma}

\begin{Remark}\label{rem:immersion}
In particular, $p_v |_X$
is an immersion for generic $v \in \PP^{2n}$.
\end{Remark}

 \begin{proof}[Proof of Lemma~\ref{lemma:goodisgenericimmersion}]
 Let $v \in \PP^{2n}$.
 We first note that $p_v |_X$ is an immersion if and only if
 $v\not\subseteq T_x X$ for all $x\in X$. Let $\P TX$ be the projectivization of
 the tangent bundle $TX$, which has ~\cite{Sri_91}dimension $2n-1$. Consider the map
 \[
 	\eta \colon\P TX \to \P^{2n} \quad \textrm{where $\eta(v) \in \PP^{2n}$
	is the line $v \subset T_x X$ viewed in $\CC^{2n+1}$} \, .
 \]
 Now, for $v \in \PP^{2n}$ we clearly have that
 $v \not \subseteq T_x X$ for all $x\in X$
 if and only if $v$ lies in the complement of $\eta(\PP TX)$.

 Let $V \subseteq TX$ be the kernel of
 the differential $D(p_l |_X) \colon TX \to T \CC^l$.
 Clearly, we have
 $\eta(\PP V) \subset \PP^{2n-l}$.
 It follows for $v \in \PP^{2n-l}$ that $p_v |_X$ is an immersion
 if and only if $v \not\in \eta(\PP V)$.
 Therefore, it is enough to show that $\dim \PP V < 2n-l$.

 For all $x \in X_i = X_i(p_l |_X)$, the fiber
 $V_x = \ker D_x(p_l |_X)$ has dimension $n - i$.
 In particular, $\PP V |_{X_n}$ is empty.
 For all $0 \leq i < n$, we have
 \[
 	\dim V |_{X_i} \leq \max_{x \in X_i} \dim V_x
	+ \dim X_i \leq (n-i) + (n + i - l) = 2n-l \, ,
 \]
 which implies $\dim \PP V |_{X_i} < 2n-l$. Hence,
 $\dim \PP V < 2n-l$ and the result follows.
 \end{proof}

\subsection{2-transversal maps}
	Let $X \subset \CC^{2n+1}$ be a smooth $n$-dimensional
	closed subvariety.
	
	\begin{lemma}\label{lemma:doublepointsllessn}
		Let $0 \leq l < n$. If $p_l|X$ is transversal, then
		$p_v |_X$ is 2-transversal
		for generic $v \in \PP^{2n-l}\subseteq \PP^{2n}$.
	\end{lemma}
		
	\begin{lemma}\label{lemma:doublepointslbiggern}
		Let $n \leq l \leq 2n-1$ and
		assume that $p_l |_X$ is proper and transversal.
		Then there exists an algebraic automorphism
		$\varphi$ of $\CC^{2n+1}$ with the following properties:
		it fixes the last $l$ coordinates and, for  generic
		$v$ in $\PP^{2n-l}\subseteq \PP^{2n}$,
		$p_v \circ \varphi |_X$ is 2-transversal.
	\end{lemma}
	
	\begin{remark}
		The conclusion of
		Lemma~\ref{lemma:doublepointslbiggern} is in general false if
		we choose $\varphi = \id$. This can be seen by the
		following simple example:
		Let $(x, y, z)$ be a coordinate system of $\CC^3$ and let
		$X \subset \CC^{3}$ be the union of three disjoint copies
		of the line $\CC$ given by
		$\{ z + y = 0, x = 1 \}$,
		$\{ x = y = 0 \}$ and $\{ z - y = 0, x = -1 \}$. Let
		$p_1 \colon \CC^3 \to \CC$ be given by
		$p_1(x, y, z) = z$.
		Thus, $p_1 |_X \colon X \to \CC$ is a trivial covering and in
		particular it is proper and transversal.
		However, $p_v |_X$
		is 2-transversal if and only if
		$v = (0: 1: 0) \in \PP^1 \subset \PP^2$.
	\end{remark}
	
	For the proof of Lemma~\ref{lemma:doublepointsllessn} and
	Lemma~\ref{lemma:doublepointslbiggern}, we consider the map
	\[
		\chi \colon X^3 \to \CC^{2n+1} \times \CC^{2n+1} \, , \
		(x, y, z) \mapsto (y-x, z-x)
	\]
	and its composition with $\pi \times \pi$
	\[
		\overline{\chi} = (\pi \times \pi) \circ \chi
		\colon
		\{ \, (x,y,z) \in X^3 \ | \ x \neq y \neq z \neq x \, \} \to
		\PP^{2n} \times \PP^{2n}
	\]
	where
	$\pi \colon \CC^{2n+1} \setminus \{ 0 \} \to \PP^{2n}$
	is the canonical projection.
	For $v \in \PP^{2n}$, note that all fibers of $p_v |_X$
	have at most two points if and only if $(v, v) \not\in \im(\overline{\chi})$.

\begin{proof}[Proof of Lemma~\ref{lemma:doublepointsllessn}]
By Lemma~\ref{lemma:transversalmultipointsdescent} there is
a non-empty open set of directions
$U$ in $\P^{2n-l}$ such that $p_v |_X$ is transversal for $v$ in $U$.
We calculate that {at} all points of $U \times U$ the map
$\overline{\chi}$ is transversal to the diagonal $\Delta$ of
$\PP^{2n} \times \PP^{2n}$;
this lengthy calculation is provided in the Appendix,
see Lemma~\ref{lemma:lengthycalc}.
Hence, {at} every point of $U \times U$ the intersection
$\im\overline{\chi} \cap \Delta$ has local dimension
less than or equal to $3n+2n-4n=n$ (see Lemma~\ref{lemma:transversality}
in the Appendix). This shows that $p_v |_X$ is
2-transversal for generic $v$ in $\P^{2n-l}$ since $0 \leq l < n$.
\end{proof}

 	\begin{proof}[Proof of Lemma~\ref{lemma:doublepointslbiggern}]		
		{
 		We claim, that $\overline{\chi}^{-1}(\PP^{2n-l} \times \PP^{2n-l})$
 		has dimension $\leq 2n-l$.	 Let $q_l \colon X \to \CC^l$
		be the restriction of $p_l \colon \CC^{2n+1} \to \CC^l$ to $X$.
		Recall, that
		$\psi_{q_l} \colon X^2 \to \CC^l$ is defined as
		$\psi_{q_l}(x, y) = q_l(y)-q_l(x)$ and consider the map
 		\[
 			\pr \colon \overline{\chi}^{-1}(\PP^{2n-l} \times \PP^{2n-l})
 			\to (\psi_{q_l})^{-1}(0) \setminus \Delta_X \, ,
			\ (x, y, z) \mapsto (x, y)
 		\]
		where $\Delta_X$ denotes the diagonal in $X \times X$.
 		As $q_l$ is transversal, it follows
 		that $0$ is a regular value of $\psi_{q_l} |_{X^2 \setminus \Delta_X}$
		and thus we get
 		$\dim (\psi_{q_l})^{-1}(0) \setminus \Delta_X \leq 2n - l$.
 		Every fiber $\pr^{-1}(x, y) = \{ (x, y, z) |
 		q_l(z) = q_l(x) \}$ is finite, since $q_l$ is proper.
 		Hence, the claim follows.
		
		Now, let $V_1, \ldots, V_r
		\subseteq \overline{\chi}^{-1}(\PP^{2n-l} \times \PP^{2n-l})$
		be the irreducible components. After applying an algebraic shear
		of $\CC^{2n+1}$ that fixes the last $l$ coordinates, we can assume
		for all $i$,
		that $\overline{\chi}(V_i)$ has points outside the diagonal
		$\Delta_{2n-l}$ in $\PP^{2n-l} \times \PP^{2n-l}$. In summary
 		\[
 			\Delta_{2n-l} \setminus
 			\overline{\overline{\chi}
			(\overline{\chi}^{-1}(\PP^{2n-l} \times \PP^{2n-l}))}
 		\]
 		is a non-empty open subset of $\Delta_{2n-l}$.
		Therefore, for generic $v\in\P^{2n-l}$, all fibers of $p_v|_X$ contain 
		at most two points. Since $p_l|_X$ is transversal,  $p_v |_X$ 
		is transversal for generic $v$ in $\P^{2n-l}$, 
		by Lemma~\ref{lemma:transversalmultipointsdescent}. 
		This implies the statement.}
 	\end{proof}

\section{Proof of the main theorem}\label{sec:proof} In this section
we prove that any two algebraic embeddings $f,g$ of a smooth
$n$-dimensional variety
$X$ in $\C^{2n+1}$ are holomorphically equivalent.

As outlined in Section~\ref{sec:strategy} we construct inductively
embeddings $f=f_0$, $f_1$, $\ldots, f_{2n}$, $f_{2n+1}=g$
interpolating $f$ and $g$.
The embedding $f_l$ is constructed to be
holomorphically equivalent to $f_{l-1}$ and to have the property
that the last $l$ coordinates are $r_l\circ g$ for some
linear projection $r_l\colon \C^{2n+1}\to \C^l$. To define these projections
$r_l$, we will choose a flag
$G= (0 \subset W_1 \subset \cdots \subset W_{2n} \subset \CC^{2n+1})$
of $\CC^{2n+1}$ and set
\[
    r_l = p_{W_{2n+1-l}}
    \colon \C^{2n+1}\to \C^{2n+1}/ W_{2n+1-l} {=} \C^l \, .
\]
Afterwards we construct the $f_l$ inductively changing $G$ generically in the
process.

In a first subsection we define the flag $G$.
Then we construct $f_1$ and prove the holomorphic
equivalence of $f_1$ and $f$. In the next subsection we
provide the inductive construction of $f_l$ using
$G$ and we prove that $f_1, \ldots, f_{2n}$
are holomorphically equivalent. In the last subsection, we prove the holomorphic
equivalence of $f_{2n}$ and $g$.

\subsection{The flag $G$.}
By Corollary~\ref{corollary:genericflags},
we can assume that there exists a non-empty
open set $O_{trans}$ in the flag variety $\F_{2n+1}$
such that $p_{W_1} \circ g$, \ldots,
$p_{W_{2n} \circ g}$ are transversal
for all flags
$(0 \subset W_1 \subset \cdots \subset W_{2n}�\subset \CC^{2n+1})$
in $O_{trans}$.

By Lemma~\ref{lemma:doublepointsllessn}, there exists a non-empty open
set $O_{2-trans} \subseteq \PP^{2n}$ with the property that
for all $w \in O_{2-trans}$ the map $p_w \circ g$ is 2-transversal.

Lemma~\ref{lemma:baselproper} provides the existence of a non-empty open
set $O_{prop}\subseteq Gr_{n+1}^{2n+1}$ with the property that for all $W$ in
$O_{prop}$ the map $p_W \circ g$ is proper.

For all $1\leq k\leq 2n$, Lemma~\ref{lemma:baselimmersion} provides
the existence of an open set $O_{imm, k} \subset \Gr_k^{2n+1}$
with the property that for all $W$ in $O_{imm, k}$
we have that $\dim X_i(p_W\circ g)\leq k+i-n-1$  for all $0 \leq i < n$.

We choose a flag
$G=(0\subset W_1\subset\cdots\subset W_{2n}\subset \C^{2n+1})$ such that
\begin{enumerate}
	\item \label{en:regular}the flag $G$ is in $O_{trans}$,
	\item \label{en:2-transversal} the $1$-dimensional subspace $W_1$
	is an element of $O_{2-trans}$,
	\item \label{en:proper} the $(n+1)$-dimensional subspace
	$W_{n+1}$ is an element of $O_{prop}$,
	\item \label{en:immersion} for all $1\leq k\leq 2n$,
	the $k$-dimensional subspace $W_k$ is an element of $O_{imm, k}$.
\end{enumerate}

Note that a generic flag in
$F_{2n+1}$ satisfies properties~\eqref{en:regular}-\eqref{en:immersion}.

\subsection{Construction of $f_1$}

We choose $v_1\in \P^{2n}$ such that
$p_{v_1}\colon \C^{2n+1}\to\C^{2n+1} / v_1$ is a good projection for $f$,
which is possible by {Corollary~\ref{corollary:baselproper}} (properness),
Remark~\ref{rem:immersion} (immersion) and
Lemma~\ref{lemma:doublepointsllessn} (2-trans\-versality).
Changing $G$ generically we
can assume that $r_1\circ g(x)\neq r_1\circ g(y)$ for all double points
$(x, y)$ of $p_{v_1}\circ f$. Hence, the map
\[	
	f_1=(p_{v_1}\circ f,r_1\circ g)\colon X\to
	\C^{2n+1} / v_1\times \C^{2n+1} / W_{2n}{=} \C^{2n}\times \C
\]
is an algebraic embedding, which is holomorphically equivalent to $f_0 = f$
by Proposition~\ref{prop:kalitrick}.

\subsection{Construction of $f_2, \ldots, f_{2n}$}

Let $1 \leq l \leq 2n-1$ and assume $f_{l}\colon X\to \C^{2n+1-l}\times \C^l$
is defined and the last $l$ coordinate functions equal $r_l \circ g$.
Let us now inductively define the embedding
$f_{l+1}$ changing the flag $G$ in the process and let us prove
the holomorphic equivalence of $f_{l+1}$ and  $f_{l}$ for $1 \leq l \leq 2n-1$.

We try to find $v_{l+1}\in \P^{2n-l}$ such that the projection
\[
	p_{v_{l+1}} \colon \C^{2n+1} \to
	\C^{2n+1} / v_{l+1} {=} \CC^{2n}
\]
is a good projection for $f_l$.
We claim that this is possible if we compose $f_l$ with some
algebraic automorphism of $\C^{2n+1}$ that fixes the last $l$
coordinate functions of $f_l$. Assuming this claim,
one can change $G = (0 \subset W_1 \subset \ldots \subset W_{2n}
\subset \CC^{2n+1})$ generically such that the subspaces
$W_{2n+1-l} \subset \cdots \subset W_{2n} \subset \CC^{2n+1}$ do not change
(i.e.\ $r_1, \ldots, r_l$ are still the same),
and $r_{l+1}\circ g$ separates the double points of $p_{v_{l+1}} \circ f_l$.
Let $e \colon \CC^{2n+1} \to \CC$ be a surjective linear function
such that $r_{l+1}$ equals
$(r_l, e) \colon \CC^{2n+1} \to \CC^l \times \CC = \CC^{l+1}$.
Now, Proposition~\ref{prop:kalitrick} shows that $f_l$ is holomorphically
equivalent to
\[
	f_{l+1}=(p_{v_{l+1}}\circ f_l, e \circ g)\colon X\to \C^{2n}\times \C.
\]
As by construction the last $l$ coordinate functions
of $f_l$ and $p_{v_{l+1}} \circ f_l$ are the same, it follows that
the last $l+1$ coordinate functions of $f_{l+1}$ equal $r_{l+1} \circ g$.

Let us prove the existence of $v_{l+1} \in \PP^{2n-l}$. In fact, we prove that
there exists an algebraic automorphism $\varphi$ of $\CC^{2n+1}$ that
fixes the last $l$ coordinates, with the following property:
for generic $v_{l+1} \in \PP^{2n-l}$ the map
$p_{v_{l+1}}$ is a good projection for
$\varphi \circ f_l$, i.e.\ it is 2-transversal, proper and immersive.

{\bf 2-transversal:}
If $l<n$, then for generic $v_{l+1} \in \PP^{2n-l}$
the map $p_{v_{l+1}} \circ f_l$ is $2$-transversal by
Lemma~\ref{lemma:doublepointsllessn}.
If $l\geq n$, then, by applying Lemma~\ref{lemma:doublepointslbiggern} to
the embedding $f_l$
(properties~\eqref{en:regular} and~\eqref{en:proper}
of $G$ imply that the assumptions of Lemma~\ref{lemma:doublepointslbiggern}
are satisfied), we find an automorphism $\varphi$ of $\C^{2n+1}$
with the following properties: $\varphi$ fixes the last
$l$ coordinates and after replacing $f_l$ by $\varphi\circ f_l$ 
the map $p_{v_{l+1}} \circ f_l$ is $2$-transversal, for
generic $v_{l+1}\in \PP^{2n-l}$.

{\bf Proper:}
If $l\leq n$, Corollary~\ref{corollary:baselproper} implies that
$p_{v_{l+1}} \circ f_l$ is proper for generic $v_{l+1} \in \PP^{2n-l}$.
If $l > n$, then $r_n \circ g$ is proper by~\eqref{en:proper}; and so of course,
$p_{v_{l+1}} \circ f_l$ is proper for all $v_{l+1} \in \PP^{2n-l}$.

{\bf Immersive:}
By applying Lemma~\ref{lemma:goodisgenericimmersion} to the embedding
$f_l$ (the assumptions of Lemma~\ref{lemma:goodisgenericimmersion}
are satisfied by~\eqref{en:immersion}),
we see that the map $p_{v_{l+1}} \circ f_l$ is immersive, 
for generic $v_{l+1}\in \PP^{2n-l}$.

\subsection{Holomorphic equivalence of $f_{2n}$ and $g$}
By construction, the last $2n$ coordinate functions of
$f_{2n}$ equal $r_{2n} \circ g$. The map $r_{2n} \circ g$ is immersive, proper
and 2-transversal
due to properties \eqref{en:immersion},
\eqref{en:proper} and \eqref{en:2-transversal} of $G$, respectively.
Thus, $f_{2n}$
and $g$ are holomorphically equivalent by Proposition~\ref{prop:kalitrick}.

\appendix

\section*{Appendix}
\renewcommand{\thesection}{A}

\begin{lemma}
	\label{lem:equidim}
	Let $X$ be a smooth variety such that every connected
	component has dimension $\leq n$ and let
	$f \colon X \to \CC^{2n+1}$ be an algebraic embedding.
	For $0 \leq i \leq n$, denote by $X_i$ the set of points $x$ in $X$
	such that $X$ is $i$-dimensional at $x$. Then there exists
	an algebraic embedding
	\[
		F \colon \coprod_{i=0}^n X_i \times \CC^{n-i}
		\to \CC^{2n+1}
	\]
	that restricts on $\coprod_{i=0}^n X_i \times \{ 0 \} = X$ to $f$.
\end{lemma}

\begin{proof}
	Let $X' \subseteq X$ be the set of points where $X$ is not $n$-dimensional.
	For generic $v \in \PP^{2n}$, the composition
	\[
		X' \stackrel{f |_{X'}}{\longrightarrow} \CC^{2n+1}
		\stackrel{p_v}{\longrightarrow} \CC^{2n+1} / v = \CC^{2n}
	\]
	is an embedding by \cite[Theorem~1']{Sri_91}. Since
	the image of the map
	\[
		X' \times X_n \to \PP^{2n} \, , \ (x', x) \mapsto [f(x') - f(x)]
	\]
	has dimension $\leq 2n-1$, for generic $v \in \PP^{2n}$,
	the two sets $p_v(f(X'))$ and $p_v(f(X_n))$ are disjoint in $\CC^{2n}$.
	In summary, there exists $w \in \CC^{2n+1}Ê\setminus \{ 0 \}$ such that
	\[
		X' \times \CC \to \CC^{2n+1} \, , \ (x', t) \mapsto f(x') + tw
	\]
	is a closed embedding and the image does not intersect $f(X_n)$.
	By a repeated application of this argument, we get our embedding $F$.
\end{proof}

\begin{lemma}
	\label{lem:prescribedjets}
	Let $m \ge 2$ be an integer, let $p_1, p_2 \in \CC^m$
	be distinct points and let $A_1, A_2 \in \SL_m$.
	Then there exists an algebraic automorphism $\varphi$ of $\CC^m$
	such that $\varphi(p_i) = p_i$ and $D_{p_i} \varphi = A_i$ for
	$i = 1, 2$. In fact, $\varphi$ can be chosen to be a composition of
	algebraic shears and linear maps.
\end{lemma}

\begin{proof}
	We can assume that $p_1 = (0, \ldots, 0)$,
	$p_2 = (1, \ldots, 1)$. For $i \neq j$ and $\lambda \in \CC$
	the algebraic shear
	\[
		\psi \colon \CC^m \to \CC^m \, , \
		(x_1, \ldots, x_m) \mapsto (x_1, \ldots, x_{i-1},
		x_i + \lambda(x_j^3-x_j^2), x_{i+1}, \ldots, x_m)
	\]
	fixes $p_1$ and $p_2$ and satisfies the following:
	$D_{p_1} \psi$ is the identity matrix
	and $D_{p_2} \psi$ is the elementary matrix with
	entry $(D_{p_2} \psi)_{ij}$ equal to $\lambda$ and all
	other entries the same as the identity matrix. As every matrix
	in $\SL_m$ is the product of elementary matrices,
	we can construct our automorphism
	$\varphi$ by composing automorphisms
	of the form $\psi$ and linear automorphisms that exchange
	$p_1$ and $p_2$.
\end{proof}

\begin{lemma}
	\label{lem:GenericSurj}
	Let $f \colon X \to Y$ be a morphism of algebraic varieties.
	Then there exists a dense open set
	$U \subseteq Y$, such that for all $x \in f^{-1}(U)$, the differential
	$D_x f \colon T_x X \to T_{f(x)} Y$ is surjective.
\end{lemma}

\begin{proof}
	If $X$ is smooth, then the result follows from the
	Generic Smoothness
	Theorem (see \cite[Corollary~10.7]{Har_77}). The general case
	follows from this result, by using the fact that we can decompose
	$X$ into finitely many smooth locally closed
	disjoint subvarieties.
\end{proof}

\begin{lemma}
\label{lemma:lengthycalc}
Let the notation be as in the proof of
Lemma~\ref{lemma:doublepointsllessn}. Then the map
$\overline{\chi}$ is transversal to $\Delta$ at all
points of $U\times U$.
\end{lemma}

\begin{proof}
Let $(x,y,z)\in X^3 $ with $\overline{\chi}(x,y,z)\in\Delta\cap (U\times U)$.
We show that $\overline{\chi}$ is transversal to $\Delta$ at $(x,y,z)$,
i.e.
\[
    \im D_{(x,y,z)}\overline{\chi}+T_{\overline{\chi}(x,y,z)}\Delta=
    T_{\overline{\chi}(x,y,z)}\P^{2n} \times \P^{2n} \, .
\]
After a linear change of coordinates in $\C^{2n+1}$ we may assume that
\[
    y-x=(1,0,\ldots,0)\quad\text{and}\quad
    z-x=(\lambda,0,\ldots,0)
\]
for some $\lambda$ in $\C\backslash\{0,1\}$.
Locally around the points $y-x$ and $z-x$ the canonical projection
$\pi\colon\C^{2n+1}\backslash\{0\}\to \P^{2n}$
looks the same as
\[
    p \colon \C^*\times\C^{2n}\to \C^{2n} \, , \ (x_0, \ldots, x_{2n})
    \mapsto \left( \frac{x_1}{x_0},\ldots,\frac{x_{2n}}{x_0} \right).
\]
More precisely, the open embedding
\[
	g\colon\C^{2n}\to\P^{2n} \, , \ (x_1,\ldots, x_{2n})
	\mapsto (1:x_1:\cdots:x_{2n})
\]
satisfies $\pi=g\circ p$ whenever $p$ is defined.
In particular, $\overline{\chi}$ is transversal to
$\Delta$ at $(x,y,z)$ if and only if the (rational) map
\[
    \widetilde{\chi}\colon X^3 \dashrightarrow
    \C^{2n}\times\C^{2n} \, , \ (a,b,c)
    \mapsto (p(b-a),p(c-a)).
\]
is transversal to the diagonal $\Delta_{\CC^{2n}}$
in $\C^{2n}\times \C^{2n}$ at $(x,y,z)$ (note that $\widetilde{\chi}$ is
defined at $(x, y, z)$).
The differential of $\widetilde{\chi}$  at $(x,y,z)$ is
\[
    \begin{pmatrix}
        -D_{x}h & D_{y}h & 0\\
        -\frac{1}{\lambda}D_{x}h & 0 &
        \frac{1}{\lambda}D_{z}h \\
    \end{pmatrix}
    \colon T_xX\times T_yX\times T_zX \to \C^{2n} \times \CC^{2n} \, ,
\]
where ${h}\colon X\to\C^{2n}$ denotes the map obtained from
the fixed embedding $X \subset \C^{2n+1}$ by forgetting
the first coordinate. Hence, the linear subspace
\[
	\im D_{(x,y,z)}\widetilde{\chi}+T_{\widetilde{\chi}(x,y,z)}\Delta_{\C^{2n}}
\]
of $\CC^{2n} \times \CC^{2n}$ is equal to
\[
    	\left\{ \,
    		\begin{pmatrix}
                		(D_{y}h)b-(D_{x}h)a+u\\
                		\frac{1}{\lambda}(D_{z}h)c - \frac{1}{\lambda}(D_{x}h)a+u
         	\end{pmatrix}
         	\ \Big| \ (a, b, c) \in T_{(x, y, z)} X^3 \, , \ u\in \C^{2n} \,
    	\right\}.
\]
Changing basis in $\C^{2n} \times \CC^{2n}$
by subtracting the first $2n$ coordinates from the last $2n$ coordinates yields
\[
	\left\{ \,
	\left(\begin{matrix} (D_{y}h)b-(D_{x}h)a+u \\
			\frac{1}{\lambda} (D_{z}h)c - \frac{1-\lambda}{\lambda}(D_{x}h)a
			-(D_{y}h)b
	       \end{matrix}
	\right)
	\ \Big| \ (a, b, c) \in T_{(x, y, z)}X^3 \, , \ u\in \C^{2n} \,
	\right\} \, .
\]
This last subspace is $\C^{2n} \times \C^{2n}$ due to
\begin{equation}
	\label{eq:diff=C2n}
	\tag{$\ast$}
	\left\{ \,
		\frac{1}{\lambda} (D_{z}h)c - \frac{1-\lambda}{\lambda}(D_{x}h)a
			-(D_{y}h)b \ \Big| \ (a, b, c) \in T_{(x, y, z)}X^3 \,
	\right\}
	=\C^{2n}.
\end{equation}
The equality~\eqref{eq:diff=C2n} follows from the fact that
\[
    D_{y}h-D_{x}h\colon T_xX\times T_yX\to\C^{2n}\quad\text{and}
    \quad D_{z}h-D_{x}h\colon T_xX\times T_zX\to\C^{2n}
\]
are surjective, which in turn follows from the fact that
$p_{\pi(y-x)} |_X$ and $p_{\pi(z-x)} |_X$ are transversal
(here we use the fact that $\overline{\chi}(x,y,z)$ is in $U\times U$).
More precisely, to prove~\eqref{eq:diff=C2n}
let $w$ be in $\C^{2n}$.
There exist $a_1,a_2$ in $T_xX$, $\tilde{b}$ in $T_yX$ and
$\tilde{c}$ in $T_zX$ such that
\[
    (D_{z}h)\tilde{c}-(D_{x}h)a_1=w=
    (D_{y}h)\tilde{b}-(D_{x}h)a_2 \, .
\]
By setting
\[
	c=\frac{\lambda}{1-\lambda}\tilde{c} \, , \quad
	b=\frac{\lambda}{1-\lambda}\tilde{b} \quad  \textrm{and} \quad
	a=\frac{\lambda}{(1-\lambda)^2}a_1-\frac{\lambda^2}{(1-\lambda)^2}a_2 \, ,
\]
we get
\[
    \frac{1}{\lambda} (D_{z}h)c - \frac{1-\lambda}{\lambda}(D_{x}h)a
			-(D_{y}h)b=w \, .
\]

Hence, $\widetilde{\chi}$ is transversal to the diagonal of $\C^{2n}\times\C^{2n}$ at $(x,y,z)$; and so, $\overline{\chi}$ is transversal to $\Delta$ at $(x,y,z)$.
\end{proof}

\begin{lemma}
	\label{lemma:transversality}
	Let $f \colon X \to Y$ be a morphism of algebraic varieties.
	If $Z$ is a closed subvariety of $Y$ such that $f$ is transversal to $Z$ at
	$z$, then the local dimension of $\im f \cap Z$ at $z$ is smaller
	than or equal to
	\[
		\max_{x \in f^{-1}(z)} \dim T_x X + \dim T_{z} Z - \dim T_{z} Y \, .
	\]
\end{lemma}

\begin{proof}
	Let $x \in f^{-1}(z)$. Since $f$ is transversal to $Z$ at $z$, the differential
	$D_x f$ induces an isomorphism
	\[
		T_x X / (D_x f)^{-1}(T_z Z) \stackrel{\sim}{\longrightarrow}
		T_z Y / T_z Z \, .
	\]
	Since $T_x f^{-1}(Z)$ lies in $(D_x f)^{-1}(T_z Z)$, we get the following
	inequality
	\[
		\dim T_x f^{-1}(Z) \leq \dim T_x X - \dim T_{z} Y + \dim T_z Z \, .
	\]
	And so the statement follows.
\end{proof}

\def\cprime{$'$} \def\cprime{$'$}
\providecommand{\bysame}{\leavevmode\hbox to3em{\hrulefill}\thinspace}
\providecommand{\MR}{\relax\ifhmode\unskip\space\fi MR }
\providecommand{\MRhref}[2]{%
  \href{http://www.ams.org/mathscinet-getitem?mr=#1}{#2}
}
\providecommand{\href}[2]{#2}

\end{document}